\documentclass[12pt,a4paper]{article}%
\usepackage{amsmath}
\usepackage{amsfonts}
\usepackage{amssymb}
\usepackage{enumerate}
\usepackage{amsthm}
\usepackage{graphicx}
\usepackage{hyperref}
\providecommand{\U}[1]{\protect\rule{.1in}{.1in}}
\newtheorem{theorem}{Theorem}

\newtheorem{corollary}[theorem]{Corollary}

\newtheorem{definition}[theorem]{Definition}

\newtheorem{lemma}[theorem]{Lemma}

\newtheorem{problem}[theorem]{Problem}
\newtheorem{proposition}[theorem]{Proposition}

\renewenvironment{proof}[1][Proof]{\noindent\textbf{#1.} }{\ \rule{0.5em}{0.5em}}

\theoremstyle{definition}
\newtheorem*{remark}{Remark}

\theoremstyle{empty}

\AtEndDocument{\bigskip{
  \noindent
  \textsc{Mikl\'os Ab\'ert, MTA Alfr\'ed R\'enyi Institute of Mathematics, Budapest, Hungary} \par  
  \noindent
  \textit{E-mail}: \href{mailto:abert.miklos@renyi.mta.hu}{\texttt{abert.miklos@renyi.mta.hu}} \par
  \addvspace{\medskipamount}
  
 \noindent
 \textsc{L\'{a}szl\'{o} M\'{a}rton T\'{o}th, Central European University and MTA Alfr\'ed R\'enyi Institute of Mathematics, Budapest, Hungary} \par  
  \noindent
  \textit{E-mail}: \href{mailto:ltoth@renyi.mta.hu}{\texttt{ltoth@renyi.mta.hu}} 

}}

\title{Uniform rank gradient, cost and local-global convergence }
\author{Mikl\'{o}s Ab\'{e}rt and L\'{a}szl\'{o} M\'{a}rton T\'{o}th\footnote{The authors were supported by the Hungarian National Research, Development and Innovation Office, NKFIH grant K109684 and the ERC Consolidator Grant 648017.}}

\begin{document}

\maketitle

\begin{abstract}
We analyze the rank gradient of finitely generated groups with respect to
sequences of subgroups of finite index that do not necessarily form a chain,
by connecting it to the cost of p.m.p.\ actions. We generalize several results
that were only known for chains before. The connection is made by the notion
of local-global convergence.

In particular, we show that for a finitely generated group $\Gamma$ with fixed
price $c$, every Farber sequence has rank gradient $c-1$. By adapting
Lackenby's trichotomy theorem to this setting, we also show that in a finitely
presented amenable group, every sequence of subgroups with index tending to
infinity has vanishing rank gradient.

\end{abstract}

\section{Introduction}

For a finitely generated group $\Gamma$ let $d(\Gamma)$ denote the minimal
number of generators (or rank) of $\Gamma$. For a subgroup $H\leq\Gamma$ of
finite index let
\[
r(\Gamma,H)=(d(H)-1)/\left\vert \Gamma:H\right\vert .
\]
The \emph{rank gradient} of $\Gamma$ with respect to a sequence $(\Gamma_{n})$
of finite index subgroups is defined to be
\[
\mathrm{RG}(\Gamma,(\Gamma_{n}))=\lim_{n\rightarrow\infty}r(\Gamma,\Gamma
_{n})
\]
when this limit exists. This notion has been introduced by Lackenby
\cite{lack} and further investigated in the literature, mainly for
\emph{chains} of subgroups. Recall that a chain in $\Gamma$ is a decreasing
sequence $\Gamma=\Gamma_{0}>\Gamma_{1}>\ldots$ of subgroups of finite index in
$\Gamma$. In this case, it is easy to see that $r(\Gamma,\Gamma_{n})$ is
non-increasing and so the limit exists.

The main goal of this paper is to offer a general framework for understanding
the rank gradient of an arbitrary sequence of subgroups in $\Gamma$ using the
cost of probability measure preserving (p.m.p.)\ actions of $\Gamma$. For chains this has been done by the first
author and Nikolov in \cite{miknik}. In \cite{torsion} arbitrary sequences
were analyzed for a special class of groups called right angled groups.

Let $(G_{n})$ be a sequence of finite graphs with an absolute degree bound. We
define the \emph{edge density} as
\[
e(G_{n})=\lim_{n\rightarrow\infty}\frac{\left\vert E(G_{n})\right\vert
}{\left\vert V(G_{n})\right\vert }%
\]
when this limit exists and the \emph{lower edge density} \underline{$e$} to be
the lim inf of the same sequence. A \emph{rewiring} of $(G_{n})$ is another
sequence of graphs $H_{n}$ on the same vertex set as of $G_{n}$, such that the
bi-Lipshitz distortion of the maps $\mathrm{id}_{V(G_{n})}$ stay bounded in $n$.
The \emph{combinatorial cost} $\mathrm{cc}(G_{n})$ is defined as the infimum
of the lower edge densities of possible rewirings of $(G_{n})$. This notion has
been introduced by Elek \cite{elek} as a discrete analogue of the notion of cost.

Our first result shows that actually combinatorial cost is more than an
analogue and, when making an additional convergence assumption, it can be
expressed as the cost of a limiting graphing.

Local-global convergence of graphs has been introduced by Bollob\'{a}s and
Riordan \cite{bollriordan} under the name partition metric, while the limiting
object and most of the known results were obtained by Hatami, Lov\'{a}sz and
Szegedy \cite{halosze}. On the group theory side, the notion is related to the
work of Kechris on weak containment, see \cite{kechris} and \cite{abelekprof}.
We postpone its definition to Section \ref{prelimsection}. For now, it
suffices to know that every sequence has a convergent subsequence and that the
limit is a graphing in the sense of Gaboriau \cite{gabor}.

\begin{theorem}
\label{fotetel}Let $G_{n}$ be a local-global convergent graph sequence. Then
we have
\[
\mathrm{cost}(\lim G_{n})=\mathrm{cc}(G_{n}).
\]
Moreover, one can choose rewirings such that the limit defining the edge density exists.
\end{theorem}

Note that similar results to Theorem \ref{fotetel} and its consequences (until Corollary \ref{cor:rankgradzero}) have been obtained independently by A. Carderi, D. Gaboriau and M. de la Salle \cite[2017]{CGS}.

This result can be effectively used to give direct proofs of results on
combinatorial cost using the already established theory of cost. For instance, Theorem \ref{fotetel}
immediately implies the following theorem of Elek \cite{elek}.

\begin{corollary}
\label{treeing}Let $G_{n}$ be a graph sequence with girth tending to infinity
such that $e(G_{n})$ exists. Then we have
\[
\mathrm{cc}(G_{n})=e(G_{n})\text{.}%
\]

\end{corollary}

Indeed, by the girth assumption, any subsequential local-global limit of
$G_{n}$ will be a so called treeing and by Gaboriau \cite{gabor}, the cost of
a treeing equals its expected degree divided by two.

In Theorem \ref{fotetel}, our graphs a priori have nothing to do with groups.
When they do come from a sofic approximation of $\Gamma$, the limiting
graphing gives rise to an essentially free probability measure preserving
action of $\Gamma$, that is unique up to weak equivalence in the sense of
Kechris \cite{kechris}. This case of local-global convergence has been
analyzed by the first author and Elek \cite{abelekpart}.

Following Gaboriau, we say that $\Gamma$ has fixed price $c$, if every
essentially free probability measure preserving action of $\Gamma$ has cost
$c$. Applying Theorem \ref{fotetel} gives us the following new result.

\begin{theorem}
\label{fixedprice}Let $\Gamma$ be a finitely generated group of fixed price
$c$. Then
\[
\mathrm{cc}(G_{n})=c
\]
for any sofic approximation $(G_{n})$ of $\Gamma$.
\end{theorem}

This gives an alternate proof of another result of Elek \cite{elek} that for
an amenable group $\Gamma$, any sofic approximation of $\Gamma$ has
combinatorial cost $1$. Indeed, by the Ornstein-Weiss theorem \cite{ornweiss},
amenable groups have fixed price $1$.

A sequence of subgroups is \textit{Farber}, if the quotient Schreier graphs
$\mathrm{Sch}(\Gamma,\Gamma_{n},S)$ form a sofic approximation of $\Gamma$. We
can now connect the cost to the rank gradient as follows.

\begin{theorem}
\label{rgrg}Let $\Gamma$ be a finitely generated group of fixed price $c$.
Then we have
\[
\mathrm{RG}(\Gamma,(\Gamma_{n}))=c-1
\]
for any Farber sequence $(\Gamma_{n})$ in $\Gamma$.
\end{theorem}

The same result is proved in \cite{miknik} for Farber \emph{chains}. Also, in
\cite{torsion} it is proved that any Farber sequence in a right angled group
has rank gradient zero. Right angled groups have fixed price $1$ by Gaboriau
\cite{gabor}, so this now immediately follows from Theorem \ref{rgrg}. Since
amenable groups also have fixed price $1$, we get the following.

\begin{corollary} \label{cor:rankgradzero}
\label{amenfarber}Let $\Gamma$ be a finitely generated amenable group. Then we
have
\[
\mathrm{RG}(\Gamma,(\Gamma_{n}))=0
\]
for any Farber sequence $(\Gamma_{n})$ in $\Gamma$.
\end{corollary}

When the sequence is not Farber, Corollary \ref{amenfarber} is clearly not
true, already for the standard lamplighter group (see \cite{abertjaikin}).
However, one can show that it still holds for \emph{finitely presented
}amenable groups.

\begin{theorem}
\label{fpamen}Let $\Gamma$ be a finitely presented amenable group. Then we
have
\[
\mathrm{RG}(\Gamma,(\Gamma_{n}))=0
\]
for any sequence $(\Gamma_{n})$ of distinct subgroups in $\Gamma$.
\end{theorem}

Behind this is the following extension of \cite{abertjaikin} that generalized
a theorem of Lackenby for normal chains \cite{lack}.

We call a sequence of finite graphs $G_{n}$ \emph{dispersive} if for any
subsequential local-global limit $\mathcal{G}$ of $G_{n}$, $\mathcal{G}$ has no strongly ergodic component of positive measure. For the notion of strong ergodicity and a graph theoretic reformulation see
Section \ref{trichotomy}.

\begin{theorem}
\label{lack}Let $\Gamma$ be a finitely presented group generated by a finite
symmetric set $S$. Let $(\Gamma_{n})$ be an arbitrary sequence of subgroups of finite
index in $\Gamma$. Then at least one of the following holds:
\begin{enumerate}[\hspace{0.5 cm}1)]
\item the sequence $\mathrm{Sch}(\Gamma,\Gamma_{n},S)$ is not dispersive;
\item $\mathrm{RG}(\Gamma,(\Gamma_{n}))=0$;
\item there exists some $n$ such that $\Gamma_{n}$ decomposes as a non-trivial amalgamated
product.
\end{enumerate}
\end{theorem}

We show that sequences in amenable groups are dispersive, and clearly they cannot decompose as a non-trivial amalgamated product. Thus Theorem
\ref{fpamen} follows as a corollary of Theorem \ref{lack}. Note that for a chain of
subgroups, being dispersive is equivalent to saying that the limiting
profinite action is not strongly ergodic. Hence, Theorem \ref{lack} implies
\cite[Theorem 3]{abertjaikin}. When the $\Gamma_{n}$ are normal in $\Gamma$, being dispersive is
equivalent to saying that $(\Gamma_{n})$ has no subsequence with Lubotzky's
property ($\tau$). So Theorem \ref{lack} also generalizes Lackenby's trichotomy
theorem \cite[Theorem 1.1.]{lack}. \bigskip

The structure of the paper is as follows. In Section \ref{prelimsection} we
define the basic notions and state some lemmas that we need for our main result. In Section \ref{section:fotetel} we prove Theorems \ref{fotetel} and \ref{fixedprice}. We introduce the analogous notions and results for group actions in Section \ref{section:groupactions} and prove Theorem \ref{rgrg}. We prove the results on finitely presented groups in Section \ref{trichotomy}. Finally, in Section
\ref{opensection} we list some open problems and suggest further directions of research.

\section{Preliminaries \label{prelimsection}}

In this section we define the basic objects of investigation of the paper and
state some known results.

\subsection{Local-global convergence}

Let $U_r$ denote the set of connected, rooted graphs with radius at most $r$ with all degrees bounded by some integer $D$. For any graph $G$, if we pick a vertex $v \in V(G)$ and look at its $r$-neighborhood $B_G(r,v)$ rooted at $v$ we get an element of $U_r$. Picking $v$ uniformly at random gives us a probability measure on $U_r$ which we will denote $P_{G,r}$, and refer to as the $r$-neighborhood statistics of $G$.

For any finite set $X$ let $M(X)$ denote the set of probability measures on $X$. We say that a sequence of graphs $(G_n)$ is \emph{locally} (or Benjamini-Schramm) convergent, if for any $r$ the sequence of probability measures $P_{G_n,r} \in M(U_r)$ converge to a limit distribution as $n \to \infty$. 

We will work with a more refined notion of convergence, and following notation from \cite{halosze} we introduce a colored version of the neighborhood statistics. Let $K(k,G)= \big\{\varphi: V(G) \to \{1,\ldots,k\}\big\}$ denote the set of $k$-colorings of the vertices of $G$. Let $U_{r}^k$ denote the set of rooted, connected, $k$-colored graphs of radius at most $r$. For any coloring $\varphi \in K(k,G)$ we can associate a colored neighborhood statistic $P_{G,r}[\varphi] \in M(U_r^k)$ as before, by choosing a uniform random vertex $v$, and then considering its $r$-neighborhood $B_G(r,v)$, this time together with the coloring $\varphi \vert_{B_G(r,v)}$.

 For $\eta_1, \eta_2 \in M(U_r^k)$ let 
\[d_{TV} (\eta_1, \eta_2) = \sup_{A \subseteq U_r^k} |\eta_1(A)- \eta_2(A)|.\] 

Note that $d_{TV}$ is the total variation distance. As we are operating in a finite dimensional space all the usual norms are equivalent.

Intuitively, a sequence of graphs $(G_n)$ is local-global convergent if for any $r,k \in \mathbb{N}$ and for $i,j$ large enough the colored neighborhood distribution $P_{G_i,r}[\varphi]$ for any $k$-coloring $\varphi$ can be approximately modeled on $G_j$, that is we can find some coloring $\psi$ such that $P_{G_i,r}[\varphi]$ and $P_{G_j,r}[\psi]$ are arbitrarily close. 

For a finite graph $G$ let $Q_{G,r}^k$ denote the finite set of possible colored neighborhood statistics arising from a graph $G$:
\[Q_{G,r}^k = \big\{P_{G,r}[\varphi] \mid \varphi \in K(k,G) \big\} \subseteq  M(U_r^k).\]

\begin{definition}
We say that a sequence of graphs $(G_n)$ is \emph{local-global convergent}  if for every $r,k \in \mathbb{N}$ the compact sets $(Q_{G_n,r}^k)$ converge in the Hausdorff distance on $\big( M(U_r^k), d_{TV} \big).$
\end{definition}

In \cite{halosze} the authors show that every sequence of bounded degree graphs has a locally-globally covergent subsequence, and that \emph{graphings} can be considered as the limit objects of convergent sequences.

\begin{definition}[\cite{halosze} Definition 3.1] \label{def:graphing}
Let $X$ be a Polish topological space and let $\mu$ be a probability measure on the Borel sets in $X$. A \emph{graphing} (with degree bound $D$) is a graph $\mathcal{G}$ on $V(\mathcal{G}) = X $ with Borel edge set $E(\mathcal{G}) \subset X \times X$ in which all degrees are at most $D$ and 
\begin{equation} \label{eqn:graphing}
\int_{A}e (x,B) \ d \mu (x) = \int_{B} e(x,A) \ d \mu (x)
\end{equation}
for all measurable sets $A,B \subseteq X$, where $e(x, S)$ is the number of edges from $x \in X$ to $S \subseteq X$.
\end{definition}

Every finite graph $G$ is a graphing with $X= V(G)$ and $\mu$ the uniform distribution on $V(G)$. 

The colored neighborhood statistics $P_{\mathcal{G},r}[\varphi]$ can easily be defined for a graphing $\mathcal{G}$, provided that the coloring $\varphi: X \to \{1, \ldots , k\}$ is chosen to be Borel. We pick a random vertex $x \in X$ according to $\mu$, and consider its colored $r$-neighborhood in $\mathcal{G}$.

As opposed to the finite case we now have to take the closure of all possible such statistics in order to obtain a compact set. Let \[Q_{\mathcal{G},r}^k = \overline{\big\{\mathcal{P}_{G,r}[\varphi] \mid \varphi: V(\mathcal{G}) \to \{1, \ldots, k\} \textrm{ Borel} \big\}}^{d_{TV}} \subseteq  M(U_r^k).\]

The graphing $\mathcal{G}$ is a local-global limit of the sequence $(G_n)$ if $Q_{G_n, r}^k \to Q_{\mathcal{G}, r}^k$ in the Hausdorff distance for all $r$ and $k$. 

We say that two graphings $\mathcal{G}$ and $\mathcal{H}$ are local-global equivalent, if the sets $Q_{\mathcal{G},r}^k$ and $Q_{\mathcal{H},r}^k$ are the same. Note that the limit is unique only up to local-global equivalence. Although we will only be dealing with sequences of finite graphs, observe that the above definition of convergence makes sense for sequences of graphings as well. 
\bigskip

\subsection{Cost}

For a graphing $\mathcal{G}$ and a vertex $x \in X$ let $[x]_{\mathcal{G}}$ denote the connected component of $x$ in $\mathcal{G}$. For two graphings $\mathcal{G}$ and $\mathcal{H}$ on the same vertex set $X$ we write $\mathcal{G} \sim \mathcal{H}$ if they have the same connected components, that is $[x]_{\mathcal{G}}=[x]_{\mathcal{H}}$ $\mu$-almost surely. 

Let $\mathcal{R}_{\mathcal{G}} \subseteq X \times X$ denote the measurable equivalence relation generated by $\mathcal{G}$, where two points are in the same equivalence class if they are in the same connected component of $\mathcal{G}$. Clearly $\mathcal{G} \sim \mathcal{H}$ if and only if $\mathcal{R}_{\mathcal{G}}=\mathcal{R}_{\mathcal{H}}$ up to measure zero. Note that every component of $\mathcal{G}$ is countable.

We will introduce a way of measuring edge sets of graphings. Let $\tilde{\mu}$ be the measure on $X \times X$ obtained the following way. For a measurable subset $C \subseteq X \times X$ let 
\[\tilde{\mu} (C) = \int_{X} \#\big\{y  \mid (x,y) \in C, y \in [x]_{\mathcal{G}}\big\} \ d \mu(x).\]
In other words on each fiber $\{x\} \times X$ we consider the counting measure concentrated on $ \{x\} \times [x]_{\mathcal{G}}$, and integrate these with respect to $\mu$ on the first coordinate.

This measure $\tilde{\mu}$ is $\sigma$-finite, it is concentrated on $\mathcal{R}_{\mathcal{G}}$, and it is easy to see that in fact it only depends on the relation $\mathcal{R}_{\mathcal{G}}$. We can similarly define $\tilde{\mu}'$ by taking the counting measures on the fibers $[x]_{\mathcal{G}}$ and integrate with respect to $\mu$ over the second coordinate. A standard argument shows that condition (\ref{eqn:graphing}) in Definition \ref{def:graphing} is equivalent to $\tilde{\mu}=\tilde{\mu}'$.

The \emph{cost} of $\mathcal{G}$ is defined to be

\[\textrm{cost} (\mathcal{G}) = \frac{1}{2} \inf \big\{ \tilde{\mu} \big( E(\mathcal{H}) \big) \ \big| \  \mathcal{H} \sim \mathcal{G}\big\}.\]
The normalization factor $\frac{1}{2}$ is included to account for counting every edge twice and to ensure coherence with \cite{gabor}. Note that the $\tilde{\mu}$ measure of the edge set of a graphing is half the expected degree of a $\mu$-random point. It is clear that if $\mathcal{H} \sim \mathcal{G}$, then their cost is the same, in fact the cost only depends on $\mathcal{R}_{\mathcal{G}}$.

The following lemma proved in \cite{gabor} sheds some light on the bi-Lipschitz condition used in the definition of combinatorial cost. We will also use it in the proof of Theorem \ref{fotetel}.

\begin{lemma}[Gaboriau] \label{lemma:bilipschitz}
Let $\mathcal{G}$ be a graphing. For every $\varepsilon > 0$ there exists some integer $L$ and some $\mathcal{H} \sim \mathcal{G}$ such that $\tilde{\mu}\big(E(\mathcal{H})\big) < \mathrm{cost}(\mathcal{G}) + \varepsilon$ and $\mathcal{G}$ and $\mathcal{H}$ are $L$-bi-Lipschitz equivalent, that is the graph metrics they define on the connected components are within a factor of $L$ from each other.
\end{lemma}

\subsection{Combinatorial cost}

The combinatorial analogue of cost for sequences of graphs is due to Elek \cite{elek}. Let $(G_n)$ be a sequence of graphs with $|V(G_n)| \to \infty$, and degree bounded by $D$. The sequence $(H_n)$ is a \emph{rewiring} of $(G_n)$ -- which we will denote $(H_n) \sim (G_n)$ -- if they have the same vertex set, and the distances defined by the graphs are uniformly bi-Lipschitz equvialent, that is $V(G_n)=V(H_n)$ and there exists some natural number $L$ such that for all $n \in \mathbb{N}$ 
\[\frac{1}{L} d_{H_n} (x,y) \leq d_{G_n} (x,y) \leq L d_{H_n} (x,y) \textrm{ for all } x,y \in V(G_n).\]
The \emph{lower edge density} of a graph sequence is defined as follows.
\[\underline{e} \big( (H_n) \big)= \liminf_{n \to \infty} \frac{|E(H_n)|}{ |V(H_n)| }.\]

\begin{definition}
The \emph{combinatorial cost} of a sequence $(G_n)$ is the infimum of the lower edge densities of its rewirings:
\[\mathrm{cc} \big( (G_n) \big) = \inf \big\{ \underline{e}\big( (H_n) \big) \ \big| \  (H_n) \sim (G_n) \big\}.\]
\end{definition}

\section{The cost of a local-global limit} \label{section:fotetel}
In this section we will prove Theorems \ref{fotetel} and $\ref{fixedprice}$.

\subsection{Proof of the main result}

We aim to show that the combinatorial cost of a locally-globally convergent graph sequence is equal to the cost of its limit. The idea of the proof is that if there is a cheap rewiring of the sequence $(G_n)$, then we can encode it into a coloring which then can be modeled with small error on the limit graphing $\mathcal{G}$. Using this coloring on the limit we can reconstruct a cheap graphing that (after some small modification) spans the same connected components as $\mathcal{G}$. In order to make this reconstruction process possible we will need to break the possible local symmetries of the graphs.

\bigskip

\begin{proof} [Proof of Theorem \ref{fotetel}]
Fix $\varepsilon > 0$ and suppose that $(H_n)$ is an $L$-rewiring of $G_n$ such that $\underline{e}(H_n) < \mathrm{cc}(G_n) + \varepsilon$. Set $r=L^2+1$, $R=2r$.

As the degrees of the $G_n$ are bounded by $D$, there is a constant $k$ such that each $G_n$ can be vertex colored by $k$ colors so that no two vertices within distance $2R$ have the same color. Fix such a coloring $\eta_n: V(G_n) \to \{1,\ldots, k\}$ for each $G_n$. The role of these $\eta_n$ is merely to break all possible symmetries of the $R$-neighborhoods.

For each vertex $v \in V(G_n)$ define its \emph{type} to be the following data. Let $(\alpha_v, \eta_v)$ denote the colored $R$-neighborhood of $v$ in $G_n$, that is $B_{G_n}(R,v)$ rooted at $v$, together with $\eta_{n}\vert_{ B_{G_n}(R,v)}$. Let $F_v$ denote the set of edges of $H_n$ that connect two vertices from $B_{G_n}(R,v)$. The type of $v$ is the triple $(\alpha_v, \eta_v, F_v)$. We think of this as a rooted, vertex colored graph with some extra distinguished edges ($F_v$) indicated. Note that $\eta_v$ assigns distinct colors to the vertices of $\alpha_v$. 

Let $T$ denote the set of all possible types. Note that $|T|$ is finite, as $k$ and $R$ are fixed. Now assigning each vertex its type can be considered as a coloring of $V(G_n)$ by $|T|$ colors. Let $\varphi_n$ denote this coloring:
\[\varphi_n : V(G_n) \to T, \quad \varphi_n(v) = (\alpha_v, \eta_v, F_v) \textrm{ for all } v \in V(G_n).\]
Observe that $\eta_n$ is a function of $\varphi_n$: for all $v \in V(G_n)$, $\eta_n(v)$ equals the color of the root of $\varphi_n$.

For any vertex $v \in V(G_n)$ the edge $(v,u) \in E(G_n)$ connecting $v$ to its neighbor $u$ can be traversed using at most $L$ edges of $H_n$. Choose a shortest path between $v$ and $u$ in $H_n$. Because the length of the steps on the edges of $H_n$ are bounded by $L$ in terms of the graph distance in $G_n$ the $r=L^2+1$-neighborhood of $v$ in $G_n$ already contains this shortest path. This holds for all neighbors $u$. This fact will be reflected in the type of $v$, namely for every neighbor of the root of $\alpha_v$ there will be a path of length at most $L$ using edges from $F_v$ connecting the root to the neighbor. We will refer to this property by saying that \emph{$F_v$ witnesses $L$-bi-Lipschitz equivalence at the root}.

Since the $G_n$ converge locally-globally to $\mathcal{G}$ if we choose $n$ large enough, we can find a Borel coloring $\varphi: X \to T$ such that \[d_{TV}\big(P_{G_n,r}[\varphi_{n}], P_{\mathcal{G},r}[\varphi]\big) < \delta,\] that is we can model the local statistics of $\varphi_n$ on $\mathcal{G}$ with at most $\delta$ error. Choose $\delta$ such that $\delta (D + \frac{1}{2} D^L) < \varepsilon$.

The type of $x$ gives a suggestion on how to construct a cheap graphing around $x$, which is $F_x$, the collection of distiguished edges. The idea is to consider the $r$-neighborhood of $x$ in $\mathcal{G}$, and choose the edges of some graphing $\mathcal{H}$ locally according to $F_x$. This $\mathcal{H}$ would have the same connected components as $\mathcal{G}$ because $F_x$ witnesses bi-Lipschitz equivalence at the root, and would be cheap because the expected $\mathcal{H}$-degree of a point is close to the expceted $F_v$-degree of the root in $P_{G_n,r}[\varphi_{n}]$. The problem is that the $r$-neighborhood of $x$ in $\mathcal{G}$ is a priori not the same as what the type of $x$ suggests it is. 

However, we will show that the above idea works for most of the points, and after a slight modification the resulting graphing will be a cheap generating graphing of the relation $\mathcal{R}_{\mathcal{G}}$. 

First we construct a Borel coloring $\eta: X \to \{1, \ldots, k\}$ from $\varphi$ imitating the way the $\eta_n$ could be recovered from the $\varphi_n$. Let $x \in X$, and let $\varphi(x)= (\alpha_x, \eta_x, F_x)$ be the type assigned to $x$ by $\varphi$. The type suggests a color for the root, namely $\eta_x (o)$ where $o$ is the root of $\alpha_x$. So we set $\eta(x) = \eta_x(o)$. Observe that $\eta_x$ is a coloring of the rooted graph $\alpha_x$, and a priori neither $\alpha_x$ nor $\eta_x$ has anything to do with the structure of $\mathcal{G}$. The value $\eta(x)$ on the other hand is a concrete color from \{1, \ldots, k\} that is assigned to the point $x \in X$. 

It will turn out that that for most points $x \in X$, their $\eta$-colored neighborhood in $\mathcal{G}$ is the same as the colored neighborhood $(\alpha_x, \eta_x)$ suggested by thier type $\varphi(x)$, and $\eta$ breaks the possible local symmetries of $\mathcal{G}$ by being injective on the neighborhood. We define $Y_1$ as the set of points where this does not hold up to distance $r$:

\[Y_1=\Big\{x \in X \ \Big| \ \big(B_{\mathcal{G}}(r,x), \eta \vert_{B_{\mathcal{G}}(r,x)}\big) \ncong \big(B_{\alpha_x}(r,o), \eta_x \vert_{B_{\alpha_x}(r,o)}\big)\Big\} \]\[ \bigcup \Big\{x \in X \ \Big| \  \eta \vert_{B_{\mathcal{G}}(r, x)} \textrm{ is not injective}\Big\}.\]

For any $x$ outside $Y_1$ we can identify $B_{\alpha_x}(r,o)$ with $B_{\mathcal{G}}(r, x)$ using their colorings. For any $v \in V(\alpha_x)$ there exists a unique $y \in B_{\mathcal{G}}(r, x)$ with $\eta(y) = \eta_x(v)$. Such a $y$ exists because of the isomorphism, and the injectivity of the colorings implies uniqueness. Later on we will denote this unique $y$ by $y_{x,v}$. The identification works the other way around as well, for every $y$ in $B_{\mathcal{G}}(r,x)$ we can find a unique $v \in V(\alpha_x)$ such that $\eta_x(v)=\eta(y)$. Let us denote this unique $v$ by $v_{x,y}$.

We use this identification to reconstruct our rewiring on $\mathcal{G}$. Define the edges of a graphing $\mathcal{H}_0$ around $x \in X \setminus Y_1$ as follows: for every edge $(o,v) \in F_x$ that connects the root $o$ of $\alpha_x$ with some other point $v \in V(\alpha_x)$ we include the edge $(x,y_{x,v})$ in $\mathcal{H}_0$.

Recall that we aim to show that $\mathcal{H}_0$ (with some small later adjustments) spans the same connected components as $\mathcal{G}$. 

\begin{definition}[Perfect points] Call a point $x \in X$ \emph{perfect}, if the following conditions hold:
\begin{enumerate}[\hspace{0.5 cm}1)]
\item $F_x$ witnesses $L$-bilipchitz equivalence at the root; 
\item all the edges in $F_x$ that $\varphi$ suggests in $B_{\alpha_x}(r,o)$ are indeed chosen to be in the edge set of $\mathcal{H}_0$.
\end{enumerate}
\end{definition}

For a type $\varphi(x)$, which is a rooted graph of radius $R=2r$ with some additional decorations we write $\varphi(x)\vert_{r}$ for the graph where we simply forget everthing outside radius $r$. Similarly when $v \in B_{\alpha_x}(r,o)$ write $\varphi(x)\vert_{r,v}$ for the rooted, decorated graph we get by considering $v$ as the root, and then forgetting everything outside radius $r$ from $v$.

\begin{definition}[Problematic points]
Let $Y_2$ be the set of points where one of the following holds.

\begin{enumerate}[\hspace{0.5 cm}i)]
\item $\varphi$ does not witness the bi-Lipschitz connectivity at the root;
\item $\varphi$ fails to capture the local $\eta$-colored structure (up to distance $R$);
\item $\eta$ is not injective up to radius $R$;
\item there is some $y$ close to $x$ where $\varphi(y) \vert_r$ differs from $\varphi(x)\vert_{r,y_x}$. 
\end{enumerate}

We call these points \emph{problematic}.
\begin{eqnarray*}
Y_2 &=& \Big\{x \in X \ \Big| \ F_x \textrm{ does not witness bi-Lipschitz equivalence at the root} \Big\}  \\
 && \bigcup  \Big\{x \in X \ \Big| \  (B_{\mathcal{G},R}(x), \eta \vert{_{B_{\mathcal{G},R(x)}}}) \ncong (\alpha_x, \eta_x)\Big\} \\
 && \bigcup \Big\{x \in X \ \Big| \  \eta \vert_{_{B_{\mathcal{G},R(x)}}} \textrm{ is not injective}\Big\} \\
 && \bigcup  \Big\{x \in X\setminus Y_1 \ \Big| \  \exists v \in B_{\alpha_x}(r,o) \textrm{ s.t. } \varphi(x)\vert_{r,v}  \ncong \varphi(v_x)\vert_{r} \Big\}.
\end{eqnarray*}
\end{definition}

The next lemma shows that because no such incoherencies happen in $G_n$ the measure of the problematic points will be small. Also note that $Y_1 \subset Y_2$, as in $Y_2$ we include all points where the local structure is not captured up to distance $R$ instead of $r$.

\begin{lemma}
The points in $(X \setminus Y_2)$ are perfect and $\mu(Y_2) < \delta.$
\end{lemma}

\begin{proof}
The $P_{G_n,r}[\varphi_n]$ and $P_{\mathcal{G},r}[\varphi]$ are probability distributions on the set $U_R^{T}$ of rooted, $T$-colored graphs of radius at most $R$, where $T$ is the set of all possible types.

Let $(\beta, o_{\beta}, \psi)$ denote such a graph with root $o_{\beta}$ and coloring $\psi: V(\beta) \to T$. For a vertex $u \in V(\beta)$ its type $\psi(u) \in T$ is the rooted $k$-colored graph $(\alpha_v, \eta_v)$ with the additional distingushed edges $F_v$. There is some root $o_{\alpha_v}$ of $\alpha_v$, and this way we define the coloring $\eta_{\psi}: V(\beta) \to \{1, \ldots, k\}$ by $\eta_{\psi}(v) = \eta_v(o_{\alpha_v})$. This $\eta_{\psi}$ is defined from $\psi$ the same way as $\eta$ (on $X$) is defined from $\varphi$. Now let $(\beta, o_{\beta}, \psi)$ be random with distribution $P_{G_n,r}[\varphi_n]$, then

\[ \mathbb{P}_{P_{G_n,R}[\varphi_n]} \left[ (\alpha_{o_{\beta}}, \eta_{o_\beta}) \cong (\beta, \eta_{\psi}) \right] = 1.\]

The above equality just restates that the type encodes the local colored structure (specifically the color of the root), but this formulation shows that this property will be inherited with small error when the total variation distance is small.

This isomorphism again enables us to identify $\alpha_{o_{\beta}}$ with $\beta$. For any $v \in V(\alpha_{o_{\beta}})$ write $y_{o_{\beta},v}$ for the unique $y \in V(\beta)$ for which $\eta_{o_{\beta}}(v) = \eta_{\psi} (y)$.

We restate that $(H_n)$ is a rewiring by saying that the distinguished edges witness the $L$-bi-Lipschitz equivalence at the root:
\[\mathbb{P}_{P_{G_n,R}[\varphi_n]} \left[ F_{o_{\beta}} \textrm{ witnesses $L$-bi-Lipschitz equivalence at the root} \right]=1.\]

We also restate the fact that the type of $v$, which is all the information up to distance $R=2r$, includes all the information in the $r$-neighborhood of some other point $u$, provided that $u$ is within distance $r$ from $v$. 
\[\mathbb{P}_{P_{G_n,R}[\varphi_n]} \left[ \psi(o_{\beta})\vert_{r,v} \cong \psi(y_{o_{\beta},v}) \vert_r \textrm{ for all } v \in V\big(B_{\alpha_{o_{\beta}}} (r, o_{\alpha_{o_{\beta}}})\big) \right] = 1.\]

Finally we restate that $\eta$ distinguishes all points in the $R$-neighborhoods.
\[\mathbb{P}_{P_{G_n,R}[\varphi_n]} \left[ \eta_{\psi} \textrm{ is injective}\right] = 1.\]

We see that the four events together hold with probability 1 with respect to $P_{G_n,R}[\varphi_n]$. Since $P_{\mathcal{G},r}[\varphi]$ is close to $P_{G_n,R}[\varphi_n]$ we get that the same holds for $\varphi$ and $\mathcal{G}$ with probability at least $1- \delta$, which implies $\mu(Y_2) < \delta$. 

If $x \in X\setminus Y_2$ and $y \in B_{\mathcal{G}}(r,x)$ then $y \in X \setminus Y_1$, which means all the distinguished edges starting from $y$ suggested by $\varphi(y)$ are indeed in $\mathcal{H}_0$, and by the definition of $Y_2$ we know that these are exactly the ones that $\varphi(x)$ would suggest. It is also clear that $F_x$ has to witness generation, as otherwise $x$ would be in $Y_2$. This implies that $x$ is perfect.
\end{proof}
\bigskip

Adding all the edges leaving the points in $Y_2$ we get $\mathcal{H}$: 

\[\mathcal{H} = \mathcal{H}_0 \cup \big\{(x,y) \in E(\mathcal{G}) \ \big| \ x \in Y_2, y \in X\big\}.\]

\begin{lemma}
$\mathcal{H}$ has the same connected components as $\mathcal{G}$, and $\tilde{\mu} \big(E(\mathcal{H})\big) \leq \frac{|E(H_n)|}{|V(H_n)|} + \varepsilon $.
\end{lemma}

\begin{proof}
$\mathcal{H}$ will have the same connected components as $\mathcal{G}$, because for any edge $(x,y) \in E(\mathcal{G})$ where $x$ is perfect the connection is witnessed by the $r$-neighborhood of $x$. If $x$ is not perfect, then $x \in Y_2$, so $(x,y) \in E(\mathcal{H})$ by definition.

We now aim to show that $\mathcal{H}$ is indeed a cheap generator for the equivalence relation.
\[\tilde{\mu}\big(E(\mathcal{H})\big) \leq \tilde{\mu} \big(E(\mathcal{H}_0)\big) + \tilde{\mu} \big(\{(x,y) \in E(\mathcal{G}) \mid x \in Y_2, y \in X\}\big) \leq \]\[ \leq \tilde{\mu} \big(E(\mathcal{H}_0)\big) + \delta D.\]

For every point $x \in X$ let $\deg_{F_x}(o)$ denote the number of edges in $F_x$ leaving the root of $\alpha_x$. It also makes sense to talk about the expectation of this $F$-degree with respect to colored neighborhood statistics, as the $F$-degree of the root can be determined from its type.
\[\tilde{\mu} \big(E(\mathcal{H}_0)\big) = \frac{1}{2} \int_{X}\ \deg_{\mathcal{H}_0}(x) \ d \mu \leq \frac{1}{2} \int_{X} \deg_{F_x}(o_{}) \ d \mu = \frac{1}{2} \mathbb{E}_{P_{\mathcal{G},r}[\varphi]} \Big[\deg_{F}(o)\Big] \leq \] \[ \leq \frac{1}{2} \left( \mathbb{E}_{P_{G_n,r}[\varphi_n]} \Big[\deg_{F}(o)\Big] + \delta D^L\right) = \frac{|E(H_n)|}{|V(H_n)|} + \frac{1}{2} \delta D^L.\]
Here we used the fact that there can be no more than $D^L$ edges leaving the root in $F_v$ (because of the bi-Lipschitz condition), and that the two distributions are close in total variation. Putting all this together and using that we chose $\delta$ to ensure that $\delta(D + \frac{1}{2} D^L) < \varepsilon$ we get 

\[\tilde{\mu} \big(E(\mathcal{H})\big) \leq \frac{|E(H_n)|}{|V(H_n)|} + \varepsilon .\]
\end{proof}
\bigskip

By the choice of $(H_n)$ we can assume that $\frac{|E(H_n)|}{|V(H_n)|} \leq \mathrm{cc}(G_n) + 2 \varepsilon$, which implies $\tilde{\mu}(E(\mathcal{H})) < \mathrm{cc}(G_n) + 3 \varepsilon$. This shows the inequality $\mathrm{cost}(\mathcal{G}) \leq \mathrm{cc}(G_n)$. 

The other inequality is proved exactly the same way. The condition that $(H_n)$ is a rewiring was only used to ensure that the bi-Lipschitz constant $L$ does not depend on $n$, only on $\varepsilon$. To prove that $\mathrm{cost}(\mathcal{G}) \geq \mathrm{cc}(G_n)$ we start by picking a cheap $L$-bi-Lipschitz generator for the single graphing $\mathcal{G}$ using Lemma \ref{lemma:bilipschitz}, and by local-global convergence we know that for $n$ large enough we can copy it to $G_n$ with small error. 

As for any large enough $n$ and $m$ the graphs $G_n$ and $G_m$ are arbitrarily close in the local-global topology we can do the same copying argument between the two. We fix the constant $L$ first, and then choose $n$ and $m$ accordingly. This shows that  (for all $L$) the rewirings $(H_n)$ can indeed be choosen such that the limits defining the edge densities exist. This finishes the proof of Theorem \ref{fotetel}.
\end{proof}

\subsection{Sofic approximations}

Using Theorem \ref{fotetel} we will show that sofic approximations of a group with fixed price $c$ have combinatorial cost $c$ as well.

\bigskip

\begin{proof}[Proof of Theorem \ref{fixedprice}]
The sequence $G_n$ of $S$-edge-labeled graphs converges to $\textrm{Cay}(\Gamma, S)$ in the Benjamini-Schramm sense, so any subsequential local-global limit will be a graphing of an essentially free action of $\Gamma$, which by the fixed price assumption implies that it has cost $c$.

First pick a locally-globally convergent subsequence $G_{n_k}$ with limit $\mathcal{G}_1$. 

\[\mathrm{cc}(G_n) \leq \mathrm{cc}(G_{n_k}) = \mathrm{cost}(\mathcal{G}_1) = c.\]

Now assume that $\mathrm{cc}(G_n) < c$. We pick an $L$ large enough such that there is some $L$-bi-Lipschitz rewiring $(H_n)$ with $\underline{e}(H_n) < c$. As $\underline{e}$ is defined by a liminf we can choose a subsequence $n_l$ such that \[\lim \frac{|E(H_{n_l})|}{|V(H_{n_l})|} < c.\]

Now by passing to a further subsequence we can assume that the $(G_{n_l})$ converge locally-globally to some $\mathcal{G}_2$. The $H_{n_l}$ witness that $\mathrm{cc}(G_{n_l}) < c$, while local-global convergence implies $\mathrm{cc}(G_{n_l}) = \mathrm{cost(\mathcal{G}_2)} = c$ by Theorem \ref{fotetel}. This is clearly a contradiction, hence $\mathrm{cc}(G_n) = c$. 
\end{proof}

\bigskip

\section{Group actions}\label{section:groupactions}

The same notions and results exist in the world of measure preserving group actions, where convergence with respect to the weak containment topology takes the place of local-global convergence. The analogous definitions and statements will be introduced in this section. 

\subsection{Groupoid cost}

Let $\Gamma$ be a finitely generated group, generated by the finite symmetric set $S=S^{-1}$. Let $(X, \mu)$  be either a standard Borel probability space or $X$ a finite set with $\mu$ the uniform measure on $X$. A probability p.m.p.\ action $f$ of $\Gamma$ is a homomorphism from $\Gamma$ to the group of measure preserving transformations of $(X,\mu)$. The image of some $\gamma \in \Gamma$ under this homomorphism will be denoted by $f_{\gamma}$. 

Any such p.m.p.\ action gives rise to a groupoid denoted $M_f$: endow $\Gamma$ with the discrete topology and counting measure, consider $M_f= X \times \Gamma$ with the product Borel structure and product measure $\tilde{\mu}$. We also define a partial product on $X \times \Gamma$: $(x_1,\gamma_1)\cdot (x_2, \gamma_2) = (x_1, \gamma_1 \gamma_2)$ whenever $x_2 = f_{\gamma_1}(x_1)$. The inverse is defined by $(x,\gamma)^{-1}=(f_{\gamma}(x),\gamma^{-1})$, and so $X \times \Gamma$ becomes a groupoid with respect to this partial product. We think of the element $(x, \gamma)$ as an arrow pointing from $x$ to $f_{\gamma}(x)$, with the arrow labeled by $\gamma$.

The notion of a generating subset of the groupoid is just as one would expect it: a subset generates, if all elements of $M_f$ can be written as a product of elements and their inverses chosen from the subset. 

For $A,B \subseteq M_f$ we will write \[A \cdot B = \{a\cdot b \mid a\in A, b \in B, \textrm{ and $a \cdot b$ is defined}\}.\] Also let $E=X \times \{e\}$, where $e$ is the identity element of $\Gamma$. Using our notation $A$ generates $M_f$ if and only if \[M_f = \bigcup_{n=1}^{\infty}(A \cup A^{-1} \cup E)^n.\]
The groupoid cost of $f$ is 
\[\textrm{gcost}(f) = \inf\{ \tilde{\mu}(A) \mid A \textrm{ generates } M_f\}.\]

The generators $S$ of the group give rise to a specific generating subset of $M_f$, namely $X_S=X \times S$.  We will say that a generating subset $A$ is $L$-bi-Lipschitz, if all the elements of $X_S$ can be generated by using at most $L$ arrows from $A$ and vice versa. More precisely we require that
\[X_S \subseteq (A \cup A^{-1} \cup E)^L, \textrm{ and } A \subseteq (X_S \cup X_S^{-1} \cup E)^L.\]
Note that while the actual value of the bi-Lipschitz constant $L$ may depend on the choice of $S$, the property of $A$ being a bi-Lipschitz generating subset (with some bi-Lipschitz constant) does not. 

In this setting Lemma \ref{lemma:bilipschitz} was stated by Ab\'ert and Nikolov \cite{miknik}. It says that by paying an arbitrarily small amount, we can choose the generating subset to be bi-Lipschitz. That is, for any $\varepsilon > 0$ there exists some integer $L$ and an $L$-bi-Lipschitz generating subset $A \subseteq M_f$ such that $\tilde{\mu}(A) < \textrm{gcost}(f) + \varepsilon$.

\subsection{The weak containment topology}

The notion of weak containment of actions was introduced by Kechris \cite{kechris}. The topology described below on the weak equivalence classes was defined by Ab\'{e}rt and Elek in \cite{abelekpart} and then studied further by Carderi in \cite{carderi}. They showed that the topology of local-global convergence is a compact topology on the weak equivalence classes of actions. 

For an action $f$ of the group $\Gamma$ and a point $x \in X$ the $\Gamma$-orbit of $x$ admits a Schreier graph structure: for two points $y,z \in \Gamma x$ in the orbit draw an oriented edge from $y$ to $z$ labeled by some $s \in S$ if $f_s(y)=z$. Denote this graph by $\textrm{Sch}(\Gamma, f, x)$. 

The only difference compared to the local-global convergence of graph sequences and graphings is that in this case we consider the neighborhoods in the Schreier graphs together with the edge labeling by the generators $S$.

To an action $f$ we again associate a set $Q_{f, r}^k$ that is the closure of all local statistics arising from Borel $k$-colorings with respect to the total variation distance. We say that an action $f$ \emph{weakly contains} another action $g$ (denoted $f \succeq g$) if $Q_{g, r}^k \subseteq Q_{f, r}^k$ for all $r$ and $k$. This means that all colorings of $g$ can be modeled on $f$ with arbitrarily small error. The actions are \emph{weakly equivalent} if they both weakly contain the other, that is $Q_{g, r}^k = Q_{f, r}^k$

Convergence with respect to the weak containment topology is defined by the convergence of $Q_{f_n, r}^k$ for all $r,k$ as compact sets with respect to the Hausdorff distance. The intuitive meaning of this convergence is the same as the one for local-global convergence. Ab\'ert and Elek showed that the topology induced by this convergence notion is compact, in particular every convergent sequence has a limit \cite{abelekpart}.

Kechris showed that if $f$ and $g$ are free p.m.p.\ actions and $f \succeq g$, then $\mathrm{cost}(\mathcal{R}_{f}) \leq \mathrm{cost}(\mathcal{R}_{g})$ \cite[Corollary 10.14]{kechris}. Here $\mathcal{R}_{f}$ denotes the orbit equivalence relation generated by the action $f$. Ab\'ert and Weiss extended this beyond free actions in \cite{abeweiss}: for any actions with $f \succeq g$ the groupoid cost satisfies $\mathrm{gcost}(f) \leq \mathrm{gcost}(g)$. This implies that the groupoid cost is well defined on weak equivalence classes, and studying the continuity properties of the groupoid cost with respect to the weak containment topology makes sense.

\subsection{The groupoid cost of weak containment limits}

\bigskip

Following the proof of Theorem \ref{fotetel} we get a result for group actions.

\begin{proposition}\label{prop:gcostsemicontinuity}
Suppose that the sequence $f_1, f_2, \ldots$ of p.m.p.\ actions is convergent in the weak containment topology to the p.m.p.\ action $f$. Then
\begin{equation}\label{eqn:semicont}
\limsup_{n \to \infty} \mathrm{gcost}(f_n) \leq \mathrm{gcost}(f).
\end{equation}
\end{proposition}

This is a semicontinuity result for the groupoid cost with respect to the weak containment topology. The proof follows exactly the same steps as in Theorem \ref{fotetel}: we choose a cheap bi-Lipschitz generating set for the groupoid $M_f$, record all local information into a coloring of $X$, modell this coloring on the $f_n$ when $n$ is large enough with some small error and build a cheap generating set for $M_{f_n}$ by decoding the coloring. 

The whole process is actually slightly easier in this setting, because there is no need to break the local symmetries of the graphs as the Schreier edge labeling already takes care of that. As we are not imposing a uniform bound on the complexity of generation by talking about ''combinatorial groupoid cost'', we only get an inequality. 

However, for the inequality we only need that colorings of $f$ can be modeled with small error on the $f_n$, and we don't have to require it the other way around. That is, if the sequence ''asymptotically weakly contains'' $f$, then we have (\ref{eqn:semicont}). This can be thought of as an asymptotic version of the monotonicity results by Kechris \cite{kechris}, and Ab\'{e}rt-Weiss \cite{abeweiss}.

\begin{remark}[\bf{The ultraproduct technique}]
These results, together with Theorem $1$ for graphings of free p.m.p.\ actions can be obtained by using the ultraproduct techniques introduced in \cite{abelekpart} and \cite{carderi}, Carderi's result on ultraproduct actions being weakly equivalent to some standard action and the monotonicity results of Kechris and Ab\'ert-Weiss. 

If one modifies (the somewhat arbitrary) choice of lower edge density in the definition of the combinatorial cost to \emph{edge density along an ultrafilter $\omega$} by taking an ultralimit instead of a liminf, then this modified combinatorial cost of the sequence will equal the cost of the ultraproduct graphing.
\end{remark}

\subsection{Rank gradient in groups with fixed price}

We need one further tool to prove Theorem \ref{rgrg}. The following lemma is stated in \cite[Lemma 21]{torsion}.

\begin{lemma} \label{lemma:gcostequalsgradient}
Let $\Gamma$ be a countable group, and $H$ a subgroup of finite index in $\Gamma$. Let $f$ be the right coset action of $\Gamma$ on $\Gamma / H$. Then we have

\[r(\Gamma, H) = \frac{\mathrm{rank}(H)-1}{|\Gamma:H|} = \mathrm{gcost} (f)-1.\]
\end{lemma}

\begin{proof}[Proof of Theorem \ref{rgrg}]
First we show that $\liminf r(\Gamma, \Gamma_n) \geq c-1$. We select a subsequence $\Gamma_{n_k}$ such that $r(\Gamma, \Gamma_{n_k})$ converges to the liminf. Taking the diagonal product of the corresponding group actions $f_{n_k}$ we get an action $f$ of $\Gamma$ that factors onto each $f_{n_k}$, which implies $\textrm{gcost}(f) \leq \textrm{gcost}(f_{n_k})$ for every $n_k$. The measure of the set of fixed points of a group element can only increase for factros, which implies that $f$ is essentially free because of the Farber condition. Using Lemma \ref{lemma:gcostequalsgradient} and $\mathrm{gcost}(f) = c$ we get

\[\lim_{k \to \infty} r(\Gamma, \Gamma_{n_k}) = \lim_{k \to \infty} (\mathrm{gcost}(f_{n_k}) - 1) \geq \mathrm{gcost}(f)-1 = c-1.  \]
Similarly we can choose a subsequence such that \[\limsup r(\Gamma, \Gamma_n) = \lim r(\Gamma, \Gamma_{n_l}).\] By passing to a further subsequence we can also assume that the actions $f_{n_l}$ converge in the weak containment topology to some action $\hat{f}$. This limit is essentially free by the Farber condition, so $\mathrm{gcost}(\hat{f}) = c$. Using Proposition \ref{prop:gcostsemicontinuity} we get

\[ \lim_{l \to \infty} r(\Gamma, \Gamma_{n_l}) =\lim_{l \to \infty} \left(\mathrm{gcost}(f_{n_l})-1 \right) \leq \mathrm{gcost}(\hat{f})-1 = c-1.\]
\end{proof}

\begin{remark}[alternative proof]
The second part of the proof can be obtained without Proposition \ref{prop:gcostsemicontinuity} -- which we only sketched -- by using a result from \cite{torsion}.

After choosing a subsequence such that $\limsup r(\Gamma, \Gamma_n) = \lim r(\Gamma, \Gamma_{n_l})$,  \cite[Theorem 8]{torsion} states that $\lim r(\Gamma, \Gamma_{n_l}) \leq \mathrm{cc} \big(\mathrm{Sch}(\Gamma, \Gamma_{n_l}, S)\big)-1$. Using Theorem \ref{fixedprice} we get  
\[ \lim_{l \to \infty} r(\Gamma, \Gamma_{n_l}) \leq \mathrm{cc} \big(\mathrm{Sch}(\Gamma, \Gamma_{n_l}, S)\big)-1 = c-1.\]
\end{remark}

\bigskip

\section{The trichotomy theorem} \label{trichotomy}

In this section we introduce strong ergodicity, and prove the results on finitely presented groups. 

\subsection{Strong ergodicity}

Let $f$ be a p.m.p.\ action of the countable group $\Gamma$ on a standard Borel space $(X,\mu)$. A sequence $A_n$ of measurable subsets is called \emph{almost invariant}, if \[\lim_{n \to \infty} \mu (f_{\gamma} A_n \triangle A_n) = 0, \textrm{ for all } \gamma \in \Gamma.\]

The action $f$ is \emph{strongly ergodic}, if for any almost invariant sequence $A_n$ we have $\lim_{n \to \infty} \mu(A_n)\big(1- \mu (A_n)\big) = 1.$ 

We will make use of the following result of Ab\'ert and Weiss \cite[Theorem 3]{abeweiss}.

\begin{theorem} \label{theorem:AW}
Let $f$ be an ergodic p.m.p.\ action of a countable group $\Gamma$ on a standard Borel space $(X, \mu)$. If $f$ is not strongly ergodic, then $f$ is weakly equivalent to $f \times I$,which is the diagonal action on $(X,\mu) \times [0,1]$ with $\Gamma$ acting trivially on the second coordinate.
\end{theorem}

\subsection{Dispersive actions}

Let $\mathrm{Sch}(\Gamma, \Gamma_n, S)$ be a sequence of Schreier graphs, and let $f_n$ denote the corresponding finite actions. We call the sequence $\mathrm{Sch}(\Gamma, \Gamma_n, S)$ \emph{dispersive} if for any
subsequential weak containment limit $f$ of $f_{n}$, $f$ has no strongly ergodic, ergodic component of positive measure.

\begin{lemma} \label{lemma:dispersive}
Let $\Gamma$ be a group generated by the finite set $S$ and $\Gamma_n$ a sequence of subgroups such that the corresponding Schreier graphs $\mathrm{Sch}(\Gamma, \Gamma_n, S)$ form a dispersive sequence. Then for every $\varepsilon > 0$ and $k \in \mathbb{N}$ we can find some $n$ such that the vertex set $V$ of $\mathrm{Sch}(\Gamma, \Gamma_n, S)$ can be partitioned into $k$ sets $A_1, \ldots, A_k$ such that
\begin{enumerate}
\item $\frac{1}{k} - \varepsilon \leq \frac{|A_i|}{|V|} \leq \frac{1}{k} + \varepsilon$ for all $A_i$,
\item $\sum |S A_i  \setminus A_i| < \varepsilon |V|.$
\end{enumerate}
\end{lemma}

\begin{proof}
Pick a subsequence $\Gamma_{n_l}$ such that the $\mathrm{Sch}(\Gamma, \Gamma_{n_l}, S)$ converge in the weak containment topology to some $\Gamma$ action $f$ on a standard Borel space $(X,\mu)$. As the sequence is dispersive we know that $f$ has no strongly egrodic, ergodic components of positive measure. 

We claim that $X$ can be partitioned into $k$ disjoined Borel sets $B_1, \ldots, B_k$ of approximately equal measure that are almost invariant, namely 
\[\sum_{s \in S} \sum_{1 \leq i \leq k} \mu(s A_i \setminus A_i) < \varepsilon/2.\]

Assume first that $f$ is ergodic. Then it is not strongly ergodic, and we can use Theorem \ref{theorem:AW}. The base space of $f \times I$ can be easily partitioned into $k$ invariant subsets, specifically $B'_i= (X,\mu) \times [(i-1)/k, i/k]$. Weak equivalence guarantees that this partition can be modeled with arbitrarily small error (e.g. $\varepsilon/2$) for $f$ on the set X, and thus we have the desired $B_i$.

When $f$ is not ergodic the we can divide $X$ into two invariant sets $X_1$ and $X_2$ of positive measure. If $f$ is ergodic on one of the $X_i$, then it is not strongly ergodic on that part and hence we can use the above argument to partion that $X_i$. If $f$ is not ergodic on $X_i$, then we can again divide it into invariant subsets of positive measure. 

Since all positive measure ergodic components are not strongly ergodic, we can continue this procedure and find a partition into some tiny (measure $\varepsilon/100)$ invariant sets and some non-strongy-ergodic components, which we can partition into almost invariant sets. Putting these blocks together into $k$ sets we can get an almost invariant partition into pieces with measure in $[1/k - \varepsilon/2, 1/k + \varepsilon/2]$. 

Now since $f$ is a limit, for some $n_l$ large enough we can model the partition $B_1, \ldots, B_k$ with error $\varepsilon/2$ on $\mathrm{Sch}(\Gamma, \Gamma_{n_l}, S)$, and get the desired $A_i$.
\end{proof}

\subsection{Finitely presented groups}

We now briefly discuss how we will present finite index subgroups of finitely presented groups. Notation and general framework follows \cite{abertjaikin}.

Let $\Gamma =  \langle S \mid R \rangle$ be a finitely presented group, and $H \subseteq \Gamma$ a finite index subgroup. Let $\mathcal{T}$ denote a spanning tree of the Schreier graph $\mathrm{Sch}(\Gamma, H, S)$. We select a transversal $T$ for the subgroup $H$ as follows: for each coset $\gamma H$ we consider the unique path in $\mathrm{Sch}(\Gamma, H, S)$ from the root $H$ to $\gamma H$, and select the corresponding $S$-word to be in $T$. This $T$ is called the left \emph{Schreier transversal} corresponding to $\mathcal{T}$ with respect to $S$. For a group element $\gamma \in \Gamma$ let $\tilde{\gamma}$ denote the unique element in $T$ such that $\gamma H= \tilde{\gamma} H$.

For every edge $e=(\gamma H, s \gamma H)$ of $\mathrm{Sch}(\Gamma, H, S)$ we put $T(e)=(\tilde{s\gamma})^{-1}s \tilde{\gamma}$. It is known that the $\{T(e)\}$ belong to and generate $H$. Note that if $e \in E(\mathcal{T})$, then $T(e)=1$.

For a relation $r=s_l \ldots s_1 \in R$ and group element $t \in T$ let $r_t= t^{-1} r t$. This $r_t$ is an element of $H$, and can be considered as a word in the $T(e)$: $r_t= T(e_l) \dots T(e_1)$, where $e_i = (s_{i-1} \ldots s_1 t H, s_{i} \ldots s_1 t H)$. We are going to use the fact that these relations give a presentation of H: \[ H = \big\langle \{T(e)\}_{e \in E(\mathrm{Sch}(\Gamma, H,S)) \setminus E(\mathcal{T})} \mid \{r_t\}_{r \in R, t \in T}\big\rangle. \]

Suppose a group $H$ has subgroups $H_i \subseteq H$ ($1 \leq i \leq k$) which all contain a fixed subgroup $L$ and $H \cong *_{L} H_i$. We say that this decomposition is non-trivial, if $L$ has index at least 3 in at least two of the subgroups.

\bigskip

\begin{proof}[Proof of Theorem \ref{lack}] 
Assume that $\mathrm{Sch}(\Gamma, \Gamma_n, S)$ is dispersive and $\mathrm{RG}(\Gamma, (\Gamma_n)) > 0$. We will show that some $\Gamma_n$ decomposes as a non-trivial amalgamated product. We can pass to a subsequence and assume that $\frac{d(\Gamma_n)-1}{[\Gamma:\Gamma_n]} > c >0$ for all $n$. Choose an integer $k$ such that 

\[\left( \frac{3}{2}|S| + 1 \right)\frac{1}{k} \leq c/2.\]

Let $M$ be the sum of the lengths of the relations in $R$. As the sequence is dispersive, using Lemma \ref{lemma:dispersive} we can choose some $n$ such that the vertex set $V\big(\mathrm{Sch}(\Gamma, \Gamma_n, S)\big)$ can be split into the disjoint union of $k$ sets $A_1, \ldots, A_k$ such that 

\begin{enumerate}
\item $\frac{[\Gamma: \Gamma_n]}{k}-\frac{[\Gamma: \Gamma_n]}{2k} < |A_j| < \frac{[\Gamma: \Gamma_n]}{k}+\frac{[\Gamma: \Gamma_n]}{2k}$ for all $j \in \{1, \ldots, k\}$ and
\item $|\partial (A_1, \ldots, A_k)| < \frac{1}{k(1+M^2)} [\Gamma: \Gamma_n]$, where 
\end{enumerate}
\[\partial (A_1, \ldots, A_k)= \big\{e \in E\big(\mathrm{Sch}(\Gamma, \Gamma_n, S)\big) \ \big| \  e=(x,y), x \in A_j, y \in A_l, j \neq l\big\}.\]

As we have $[\Gamma:\Gamma_n] \to \infty$ we can also make sure that we choose $n$ large enough so that $\frac{k-1}{[\Gamma:\Gamma_n]} \leq c/2$. Put $H=\Gamma_n$ and follow the above construction for a presentation of $H$. Note that $\big|V\big(\mathrm{Sch}(\Gamma, H, S)\big)\big|=[\Gamma : H]$. 

Define $Y$ to be the collection of generators $T(e)$ that share a relation with an inbetween edge. More precisely let $T(e) \in Y$ if either $e \in \partial(A_1, \ldots, A_k)$ or there exists a relation $r_t=T(e_1)^{\pm 1}\dots T(e_l)^{\pm 1}$ for which some $e_j \in \partial(A_1, \ldots, A_k)$ and some $e_m = e$. Let $X_i$ be the set of generators $T(e)$ for which both endpoints of $e$ are in $A_i$. 

Define the subgroups $L = \langle Y \rangle$ and $H_i = \langle Y \cup X_i \rangle$. Clearly $L \leq H_i$ for all $i$. 

\begin{lemma} \label{lemma:amalgamated}
$H$ decomposes as the amalgamated product of the $H_j$ over $L$,  $H \cong *_{L} H_i$.
\end{lemma}
We postpone the proof of the lemma, and show that this decomposition is non-trivial. Suppose that $L$ has index at most 3 in $H_1, H_2, \ldots, H_{k-1}$. Then each of these $H_i$ ($1 \leq i \leq k-1)$ are generated by $L$ and at most 1 other element. Thus $H$ is generated by $H_k$ (which includes $L$) and at most $k-1$ other elements, $d(H) \leq d(H_k) + k-1$.

It is easy to bound the cardinality of $|X_k|$:

\[|X_k| \leq |S||A_k| \leq |S|\left(\frac{[\Gamma : H]}{k}+\frac{[\Gamma : H]}{2k}\right).\]

Let $r=s_l \ldots s_1$ be a relation of $\Gamma$ of lenght $l$. Note that there are at most $l |\partial(A_1, \ldots, A_k)|$ different lifts $r_{\tilde{g}}=T(e_1)^{\pm 1} \dots T(e_l)^{\pm 1}$ of $r$ for which some $e_j \in \partial(A_1, \ldots, A_k)$. Also for each such relation of $H$ we have at most $l$ generators $T(e)$ of $H$ that are getting into $Y$. Thus, if $\{l_j\}$ is the set of lenghts of the relations of $\Gamma$ (so $M= \sum l_j$), then we can bound the cardinality of $Y$.  

\[|Y| \leq |\partial(A_1, \ldots, A_k)| (1 + \sum l_j^2) \leq |\partial(A_1, \ldots, A_k)| (1 + M^2) \leq \frac{[\Gamma : H]}{k}.\]
 Putting our bounds together we get 

\[d(H) \leq d(H_k) + k-1 \leq |X_k| + |Y| + k-1 \leq \left( \frac{3}{2}|S| + 1 \right) \frac{[\Gamma : H]}{k} +k-1.\]
This however gives an upper bound on the rank quotient at $H$:

\[\frac{d(H)-1}{[\Gamma: H]} \leq \left( \frac{3}{2}|S| + 1 \right)\frac{1}{k} + \frac{k-1}{[\Gamma:H]} \leq c/2 + c/2 =c.\]
This contradicts our assumption that each such quotient is more than $c$, hence the decomposition is non-trivial.
\end{proof}
\bigskip

\begin{proof}[Proof of Lemma \ref{lemma:amalgamated}]
The argument follows the one in \cite[Section 3]{abertjaikin}.

Consider the following sets of relations. Let $R_i$ be the set of all the $r_t=T(e_1)^{\pm 1}\dots T(e_l)^{\pm 1}$ where either all $e_j$ have both endpoints in $A_i$ or some $e_j$ is in $\partial(A_1, \ldots, A_k)$. Now $R_i \cup R_j$ is the same set $\bar{R}$ for all pairs $(i,j)$, that is the relations having an inbetween edge. Define the groups $T_i$ by the presentations $\langle X_i \cup Y \mid R_i \rangle$, let $\bar{T} = \langle  Y \mid \bar{R} \rangle$. We have a homomorphisms $\phi_j:\bar{T} \to T_i$ by the inclusion of $Y$ into $X_i \cup Y$. From the presentations we see that $H \cong *_{\bar{T}} T_i$.

Each $T_i$ surjects onto $H_i$ (and $\bar{T}$ surjects onto $L$) by mapping the abstract generators to their counterparts in $H$. By the universal property of the amalgamated product one can see $H \cong *_{L} H_i$.
\end{proof}

\begin{proposition}
For a countable amenable group $\Gamma$ all sequences $(\Gamma_n)$ of distinct finite index subgroups are dispersive.
\end{proposition}

\begin{proof}
It is a result of Schmidt \cite[Theorem 2.4]{Schmidt} that amenable groups admit no strongly ergodic actions. If any subsequential limit would have an ergodic component of positive measure that is strongly ergodic, then restricting the action to that component would contradict \cite{Schmidt}. 
\end{proof}

As amenable groups cannot decompose as non-trivial amalgamated products this proves Theorem \ref{fpamen}.

\section{Open problems \label{opensection}}

Note that an even stronger form of Theorem \ref{rgrg} would simply express the
rank gradient of a local-global convergent sequence as the cost of the
limiting graphing of the sequence. As of this moment, we do not know how to
prove or disprove this. The obstacle is that combinatorial cost handles
sequences of generating sets with a bounded complexity with respect to some
standard generating set. A priori, it could happen that actual small
generating sets over the sequence need a very aggressive growth of complexity,
and the local-global metric is too weak to connect such generating sets over
the sequence. This is the same obstacle that makes the proof of Theorem
\ref{lack} somewhat tricky.

The following problem connects two well-known unsolved problems, one in
ergodic theory, the other in $3$-manifold theory.

\begin{problem}
\label{fixrg}Let $\Gamma$ be a finitely generated group. Does there exist $c$
such that for any Farber sequence $(\Gamma_{n})$ in $\Gamma$, we have
$\mathrm{RG}(\Gamma,(\Gamma_{n}))=c$?
\end{problem}

Equivalently, one can ask whether $\mathrm{RG}(\Gamma,(\Gamma_{n}))$ exists
for any Farber sequence $(\Gamma_{n})$ in $\Gamma$. Indeed, an advantage of
Farber sequences over Farber chains is that they are closed to merging.

By Theorem \ref{rgrg}, a negative answer to Problem \ref{fixrg} would
immediately give a negative answer to the Fixed Price problem of Gaboriau
\cite{gabor}, that asks whether for an arbitrary countable group $\Gamma$, all
essentially free p.m.p.\ actions of $\Gamma$ have the same cost. A positive
answer, on the other hand, would specifically show that in a finitely
generated group, any two normal chains with trivial intersection have the same
rank gradient, which by \cite{miknik} would then solve the strong Rank vs
Heegaard genus problem on hyperbolic $3$-manifolds.

One possible approach to Problem \ref{fixrg} is through graph theory as
follows. Let $G$ be a finite, connected graph with maximal degree $D$. For
$L\geq1$ let
\[
c_{L}(G)=\min_{H}\frac{\left\vert E(H)\right\vert }{\left\vert V(G)\right\vert
}%
\]
where $H$ runs through all rewirings of $G$ with bi-Lipschitz constant at most
$L$. It is easy to see that
\[
1-o(1)\leq c_{L}(G)\leq D\text{.}%
\]

The following problem is related to the Fixed Price problem of Gaboriau.

\begin{problem}
\label{problema}Let $(G_{n})$ be a Benjamini-Schramm convergent sequence of
graphs of bounded degree. Does $c_{L}(G_{n})$ converge for every $L\geq1$?
\end{problem}

The connection is one sided.

\begin{proposition}
\label{kapcsolat}An affirmative solution of Problem \ref{problema} implies an
affirmative solution of Problem \ref{fixrg}.
\end{proposition}

Note, however, that this problem seems to be a real strenghtening. Indeed, it
could happen that for two large graphs $G_{1}$ and $G_{2}$ that are very close
in the Benjamini-Schramm topology, one can find a cheap rewiring of $G_{1}$ with a
bi-Lipschitz constant $L_{1}$ but only do the same to $G_{2}$ with a much
bigger constant.

\bigskip

\begin{problem}
Let $\Gamma$ be a finitely presented group generated by a finite symmetric set
$S$. Let $(\Gamma_{n})$ be a sequence of subgroups of finite index in $\Gamma$
and let $G_{n}=\mathrm{Sch}(\Gamma,\Gamma_{n},S)$. Assume that no $\Gamma_{n}$
decomposes as a non-trivial amalgamated product and that the sequence $(G_{n})$ is dispersive. Is it true that $\mathrm{cc}%
(G_{n})=1$?
\end{problem}

In Theorem \ref{lack} we show that the rank gradient of $(\Gamma_{n})$ must vanish.

\end{document}